\documentclass[12pt,reqno]{amsart}

\usepackage{amssymb}

\newtheorem{theorem}{Theorem}[section]
\newtheorem{proposition}[theorem]{Proposition}
\newtheorem{lemma}[theorem]{Lemma}
\newtheorem{corollary}[theorem]{Corollary}

\theoremstyle{definition}
\newtheorem{definition}[theorem]{Definition}

\numberwithin{equation}{section}

\begin{document}
\baselineskip=15.5pt

\title[Finite vector bundles over compact complex manifolds]{\'Etale triviality of
finite vector bundles over compact complex manifolds}

\author[I. Biswas]{Indranil Biswas}

\address{School of Mathematics, Tata Institute of Fundamental
Research, Homi Bhabha Road, Mumbai 400005, India}

\email{indranil@math.tifr.res.in}



\date{}

\begin{abstract}
A vector bundle $E$ over a projective variety $M$ is called finite if it satisfies a 
nontrivial polynomial equation with nonnegative integral coefficients. Introducing finite 
bundles, Nori proved that $E$ is finite if and only if the pullback of $E$ to some finite 
\'etale covering of $M$ is trivializable \cite{No1}. The definition of finite bundles extends 
naturally to holomorphic vector bundles over compact complex manifolds. We prove that a 
holomorphic vector bundle over a compact complex manifold $M$ is finite if and only if the 
pullback of $E$ to some finite \'etale covering of $M$ is holomorphically trivializable.
Therefore, $E$ is finite if and only if it admits a flat holomorphic connection with finite
monodromy. In \cite{BP} this 
result was proved under the extra assumption that the compact complex manifold $M$ admits a 
Gauduchon astheno-K\"ahler metric.
\end{abstract}

\maketitle

\section{Introduction}

Given a vector bundle $E$ on a projective variety and a polynomial $P(x)\,=\, \sum_{i=0}^N 
a_i x^i$, where $a_i$ are nonnegative integers, define
$$
P(E)\,:=\, \bigoplus_{i=0}^N (E^{\otimes i})^{\oplus a_i}\, ,
$$
where $E^{\otimes 0}$ is the trivial line bundle. A vector bundle $E$ is called finite if 
there are two distinct polynomials $P_1$ and $P_2$ of the above type such that the two vector 
bundles $P_1(E)$ and $P_2(E)$ are isomorphic; finite bundles were introduced by Nori in 
\cite{No1}. The above definition clearly makes sense if $E$ is a holomorphic vector bundle on 
a compact complex manifold.

There are several equivalent definitions of finiteness \cite[p.~35, Lemma 3.1]{No1}. The one 
that is most useful to work with is the following: A vector bundle $E$ is finite if and only 
if there are finitely many vector bundles $V_1,\, \cdots,\, V_p$ such that
$$
E^{\otimes j}\, =\, \bigoplus_{i=1}^p V^{\oplus j_i}_i
$$
for all $j\, \geq\, 1$, where $j_i$ are nonnegative integers.

Nori proved that a vector bundle $E$ on a complex projective variety $Z$ is finite if and 
only if the pullback of $E$ to some finite \'etale covering of $Z$ is trivializable 
\cite{No1}, \cite{No2}. In \cite{No1}, \cite{No2} the base field is allowed to be any 
algebraically closed field, although we have stated here for complex numbers. When $Z$ is 
smooth, a vector bundle $E$ pulls back to the trivial bundle under some finite \'etale 
covering of $Z$ if and only if $E$ admits a flat holomorphic connection with finite 
monodromy.

It is natural is ask whether the above characterization of finite vector
bundles remains valid for holomorphic vector bundles over compact complex manifolds.

In \cite{BHS} this characterization of finite bundles was proved for holomorphic
vector bundles on compact K\"ahler manifolds. In \cite{BP} this characterization was
proved for holomorphic vector bundles on compact complex manifolds that admit a
Gauduchon astheno-K\"ahler metric.

The following is proved here (Theorem \ref{thm2} and Corollary \ref{cor2}):

\begin{theorem}\label{thm0}
Let $E$ be a holomorphic vector bundle over a compact complex manifold $M$. Then $E$ is finite 
if and only if the pullback of $E$ to some finite \'etale covering of $M$ is holomorphically 
trivializable. Also, $E$ is finite if and only if it admits a flat holomorphic connection 
with finite monodromy.
\end{theorem}

Both \cite{BHS} and \cite{BP} used, very crucially, the numerically flat bundles introduced 
by J.-P. Demailly, T. Peternell and M. Schneider in \cite{DPS}. Finite bundles are 
numerically flat. Theorem 1.18 of \cite[p.~311]{DPS} proves a characterization of numerically 
flat bundles on compact K\"ahler manifolds. This theorem of \cite{DPS}
was directly used in \cite{BHS}, and 
it was extended in \cite[p.~5, Theorem 3.2]{BP} to numerically flat bundles on compact 
complex manifolds admitting a Gauduchon astheno-K\"ahler metric, which was then used in an 
essential way. The proof of Theorem \ref{thm0} avoids use of numerically flat bundles.

\section{Holomorphic bundles on compact complex manifolds}

\subsection{Finite vector bundles}

Let $M$ be a compact connected complex manifold. A holomorphic vector bundle $E$ over $M$ is 
called \textit{decomposable} if there are holomorphic vector bundles $V$ and $W$ over $M$ of 
positive ranks such that $V\oplus W$ is holomorphically isomorphic to $E$. A holomorphic 
vector bundle is called \textit{indecomposable} if it is not decomposable. Clearly, any 
holomorphic vector bundle can be expressed as a direct sum of indecomposable vector bundles. 
Atiyah proved the following uniqueness theorem for such decompositions.

\begin{theorem}[{\cite[p.~315, Theorem~3]{At}}]\label{thm1}
Let $E$ be a holomorphic vector bundle over $M$ holomorphically isomorphic to both
$\bigoplus_{i=1}^m V_i$ and $\bigoplus_{j=1}^n W_j$, where $V_i$,
$1\, \leq\, i\, \leq\, m$, and $W_j$, $1\, \leq\, j\, \leq\, n$, are all
indecomposable vector bundles. Then $m\,=\, n$, and there is a permutation
$\sigma$ of $\{1,\, \cdots,\, m\}$ such that $V_i$ is holomorphically isomorphic to
$W_{\sigma(i)}$ for all $1\, \leq\, i\, \leq\, m$.
\end{theorem}

\begin{definition}\label{def0}
A holomorphic vector bundle $V$ will be called a \textit{component} of $E$ if there
is another holomorphic bundle $W$ such that $V\oplus W$ is isomorphic to $E$. If
a component $V$
is also indecomposable, then it would be called an \textit{indecomposable component}.
\end{definition}

A holomorphic vector bundle $E$ on $M$ is called \textit{finite} if there are finitely many
holomorphic vector bundles $V_1,\, \cdots,\, V_p$ such that
\begin{equation}\label{e1}
E^{\otimes j}\, =\, \bigoplus_{i=1}^p V^{\oplus j_i}_i
\end{equation}
for all $j\, \geq\, 1$, where $j_i$ are nonnegative integers (\cite[p.~35, Definition]{No1}, 
\cite[p.~35, Lemma 3.1(d)]{No1}), \cite[p.~80, Definition]{No2}. As mentioned in the introduction,
$E$ is finite if and only if there are two distinct polynomials $P_1$ and $P_2$, whose coefficients
are nonnegative integers, such that $P_1(E)$ is holomorphically isomorphic to $P_2(E)$
\cite[p.~35, Definition]{No1}, \cite[p.~35, Lemma 3.1(a)]{No1}.

Any vector bundle of the
form $\bigoplus_{i=1}^N W^{\oplus n_i}_i$, where $n_i$ are nonnegative integers, will be called
a direct sum of copies of $W_1,\, \cdots,\, W_N$.

We can assume that each $V_i$ in the above definition is a component (see Definition \ref{def0}) of 
some $E^{\otimes m}$. If $V_i$ is a component of $E^{\otimes m}$, then $V^{\otimes k}$ is a 
component of $E^{\otimes km}$ for all $k\, \geq\, 1$. Hence using Theorem \ref{thm1} it follows 
that $V^{\otimes k}$ is a direct sum of copies of the indecomposable components (see Definition 
\ref{def0}) of $\bigoplus_{i=1}^p V_i$. So $V_1,\, \cdots,\, V_p$ are finite bundles. Similarly, it 
follows from Theorem \ref{thm1} that any component of a finite bundle is also a finite bundle 
\cite[p.~36, Lemma 3.2(2)]{No1}, \cite[p.~80, Lemma 3.2(2)]{No2}.

Hence we may assume that each $V_i$ in \eqref{e1} is indecomposable.

If $E_1$ and $E_2$ are finite vector bundles, then both $E_1\oplus E_2$ and $E_1\otimes E_2$ are 
also finite \cite[p.~36, Lemma 3.2(1)]{No1}, \cite[p.~80, Lemma 3.2(1)]{No2}. A holomorphic line 
bundle is finite if and only if it is of finite order \cite[p.~36, Lemma 3.2(3)]{No1}, \cite[p.~80, 
Lemma 3.2(3)]{No2}. If $E$ is finite then clearly the dual bundle $E^*$ is also finite.

\subsection{Semistable sheaves}

Let $d$ be the complex dimension of $M$.
A Gauduchon metric on $M$ is a Hermitian structure $g$ on $M$
such that the associated positive $(1,\, 1)$--form $\omega_g$ on $M$ satisfies
the equation $$\partial\overline{\partial}\omega^{d-1}_g\,=\, 0\, .$$
Gauduchon metric exists \cite[p.~502, Th\'eor\`eme]{Ga}. Fix a Gauduchon
metric $g$ on $M$. This enables us to define the degree of a torsionfree coherent
analytic sheaf $F$ on $M$ as follows: Fix a Hermitian structure on the determinant line bundle
$\det(F)\, \longrightarrow\, M$ (see \cite[Ch.~V, \S~6]{Ko2} for determinant bundle), and
denote by ${\mathcal K}$ the curvature of the corresponding Chern connection; now define
$$
\text{degree}(F)\,=\, \frac{1}{2\pi\sqrt{-1}}\int_M {\mathcal K}\wedge \omega^{d-1}_g
\, \in\, \mathbb R
$$
\cite[p.~43, Definition 1.4.1]{LT}. Define the slope of $F$
$$
\mu(F)\, :=\, \frac{\text{degree}(F)}{\text{rank}(F)}\, \in\, \mathbb R\, .
$$
This allows us to define stability and semistability \cite[p.~44, Definition 1.4.3]{LT}.
We recall that a torsionfree coherent
analytic sheaf $F$ on $M$ is called polystable if it is a direct sum of stable sheaves
of same slope; in particular, polystable sheaves are semistable.

For a torsionfree coherent analytic sheaf $F$ on $M$, define $\mu_{\rm max}(F)$ to be the slope 
of the maximal semistable subsheaf of $F$. In other words, $\mu_{\rm max}(F)$ is the slope of 
the first term in the Harder--Narasimhan filtration of $F$.

A holomorphic vector bundle on $M$ is polystable if it admits a Hermitian--Einstein metric 
\cite[p.~55, Theorem 2.3.2]{LT} (see also \cite[p.~12, Proposition 4]{Lu}, \cite[p.~177--178, 
Theorem 8.3]{Ko2} and \cite{Ko1}). Moreover, A stable holomorphic vector bundle on $M$ admits 
a Hermitian--Einstein metric \cite[p.~563, Theorem 1]{LY}, \cite[p.~61, Theorem 3.0.1]{LT}. 
If $H_1$ and $H_2$ are Hermitian--Einstein metrics on $E$ and $F$ respectively, then 
$H_1\otimes H_2$ is a Hermitian--Einstein metric on $E\otimes F$. This immediately implies 
that the tensor product of two stable vector bundles is polystable. Now it is straight-forward to 
deduce the following:

\begin{corollary}\label{cor1}
If $E$ and $F$ are polystable vector bundles, then $E\otimes F$ is also polystable.
\end{corollary}

For the following proposition see \cite[p.~22, Proposition 1.5.2]{HL} (in \cite[Proposition 
1.5.2]{HL} Gieseker (semi)stability is used, but the proof works for $\mu$--(semi)stability; see 
\cite[p.~25, Theorem 1.6.6]{HL} for it), \cite[p.~175--176, Theorem 7.18]{Ko2}, \cite[p.~1034, 
Proposition 3.1]{BDL}, \cite[p.~998, Proposition 2.1]{BTT}.

\begin{proposition}\label{prop1}
Let $E$ be a semistable holomorphic vector bundle on $M$, and let
$$
0\,=\, F_0\, \subset\, F_1\, \subset\, \cdots \, \subset\, F_{b-1} \, \subset\, F_{b}\,=\, E
$$
be a filtration of reflexive subsheaves such that $F_i/F_{i-1}$ is stable with
$\mu(F_i/F_{i-1})\,=\, \mu(E)$ for all $1\, \leq\, i\, \leq\, b$. Then the holomorphic isomorphism
class of the graded object
$$
\bigoplus_{i=1}^b F_i/F_{i-1}
$$
is independent of the choice of the filtration.
\end{proposition}

\begin{definition}\label{def1}
The stable sheaves $F_i/F_{i-1}$, $1\, \leq\, i\, \leq\, b$, in Proposition \ref{prop1}
will be called the \textit{stable graded factors} of $E$.
\end{definition}
 
\section{Homomorphisms of finite bundles}

\begin{lemma}\label{lem-1}
Let $E$ be a finite vector bundle on $M$ and
$$
\phi\, \in\, H^0(M,\, E)\setminus \{0\}
$$
a nonzero holomorphic section. Then $\phi$ does not vanish at any point of $M$.
\end{lemma}

\begin{proof}
Assume that $\phi$ vanishes at a point $x_0\, \in\, M$. Then
$$
\phi^{\otimes j}\, \in\, H^0(M,\, E^{\otimes j})
$$
vanishes at $x_0$ of order at least $j$.

Take any holomorphic vector bundles $A$ on $M$. For any nonzero holomorphic
section $h\, \in\, H^0(M,\, A)\setminus\{0\}$, let ${\rm Ord}_A(h,x_0)$ denote the order of 
vanishing of $h$ at $x_0$. It can be shown that the set of all such integers
$$\{{\rm Ord}_A(h,x_0)\}_{h\in H^0(M, A)\setminus\{0\}}$$ is bounded above. Indeed, the
space of holomorphic sections of $A$ vanishing at $x_0$ of order at least $k$ is a subspace
of the finite dimensional vector space $H^0(M,\, A)$, and $H^0(M,\, A)$ is filtered by such
subspaces.

Consider the decomposition of $E^{\otimes j}$ in \eqref{e1}; fix an isomorphism of
$E^{\otimes j}$ with the direct sum. Take any
$1\, \leq\, k \, \leq\, p$ such that
\begin{enumerate}
\item $j_k\, >\, 0$, and

\item the projection of the section $\phi^{\otimes j}$ to some component $V_k$
of $E^{\otimes j}$ (of the $j_k$ components) is nonzero.
\end{enumerate}
Since $\phi^{\otimes j}$
vanishes at $x_0$ of order at least $j$, the projection of $\phi^{\otimes j}$ to any
component of $E^{\otimes j}$ vanishes at $x_0$ of order at least $j$ (if the projection
does not vanish identically). But this contradicts the above observation that the
orders of vanishing, at $x_0$, of the nonzero sections of a given holomorphic vector
bundle are bounded above. This proves that $\phi$ does not vanish at any point of $M$.
\end{proof}

\begin{proposition}\label{prop2}
Let $E$ and $F$ be finite holomorphic vector bundles on $M$, and let
$$
\Phi\, :\, E\, \longrightarrow\, F
$$
be an ${\mathcal O}_M$--linear homomorphism. Then the image $\Phi(E)$ is a subbundle of $F$.
\end{proposition}

\begin{proof}
Let $r$ be the rank of $\Phi(E)$. We assume that $r\, >\, 0$, because the proposition is
obvious for $r\,=\, 0$. Consider the holomorphic vector bundle
$$
\text{Hom}(\bigwedge\nolimits^r E, \, \bigwedge\nolimits^r F)\,=\,
(\bigwedge\nolimits^r F)\otimes (\bigwedge\nolimits^r E)^*\, .$$
Since $F$ is a finite vector bundle, we know that $\bigotimes^r F$ is also finite. Therefore,
the direct summand $\bigwedge^r F$ of $\bigotimes^r F$ is also finite. Similarly,
$(\bigwedge^r E)^*$ is also finite. Consequently, $\text{Hom}(\bigwedge\nolimits^r E, \,
\bigwedge\nolimits^r F)$ is a finite bundle.

Let
$$
\widehat{\Phi}\, :=\, \bigwedge\nolimits^r \Phi\, \in \, H^0(M,\, \text{Hom}(
\bigwedge\nolimits^r E, \, \bigwedge\nolimits^r F))
$$
be the homomorphism of $r$--th exterior products corresponding to $\Phi$. Since the
rank of $\Phi(E)$ is $r$, the rank of the
subsheaf $\widehat{\Phi}(\bigwedge\nolimits^r E)\, \subset\,
\bigwedge\nolimits^r F$ is one, and the section
$\widehat{\Phi}$ vanishes exactly on the closed subset of $M$ over which $\Phi(E)$ fails to be
a subbundle of $F$. But $\widehat{\Phi}$ does not vanish anywhere by Lemma \ref{lem-1}.
Consequently, $\Phi(E)$ is a subbundle of $F$.
\end{proof}

\begin{lemma}\label{lem1}
Let $E$ be a finite vector bundle on $M$. Then $E$ is semistable, and
$$
{\rm degree}(E)\,=\, 0\, .
$$
\end{lemma}

\begin{proof}
Consider the decomposition in \eqref{e1}. We have
$$
\mu(\bigoplus_{i=1}^p V^{\oplus j_i}_i)\,=\, \sum_{i=1}^p \frac{j_i\cdot
\text{rank}(V_i)}{\sum_{i=1}^p j_i\cdot \text{rank}(V_i)} \mu(V_i)\, ,
$$
in particular, $\text{Min}\{\mu(V_i)\}_{i=1}^p\, \leq\,
\mu(\bigoplus_{i=1}^p V^{\oplus j_i}_i)\, \leq\, \text{Max}\{\mu(V_i)\}_{i=1}^p$.
Hence from \eqref{e1} it follows that $\{\mu(E^{\otimes j})\}_{j=1}^\infty$ is
bounded. On the other hand,
$$
\mu(E^{\otimes j})\,=\, j\cdot \mu(E)\, .
$$
Therefore, we have $\mu(E)\,=\, 0$. This implies that ${\rm degree}(E)\,=\, 0$.

If $E$ is not semistable, take any subsheaf $W\, \subset\, E$ such that $\mu(W)\, >\, 
\mu(E)\,=\, 0$. So $\det (W)$ is a subsheaf of $\bigwedge^s E \, \subset \, \bigotimes^s E$, 
where $s$ is the rank of $W$. Hence the line bundle $(\det (W))^{\otimes j}$ is a subsheaf of 
$\bigotimes^{s\cdot j} E$.
Now from \eqref{e1} it follows that $(\det (W))^{\otimes j}$ is a subsheaf of some $V_i$
for some $i\, \in\, \{1, \cdots,\, p\}$ (the projection of $(\det (W))^{\otimes j}$ to one
of the direct summands in \eqref{e1} has to be nonzero). This implies that
$$
\mu((\det (W))^{\otimes j})\,\leq\, \text{Max}\{\mu_{\rm max}(V_1),\, \cdots,\,
\mu_{\rm max}(V_p)\}\, .
$$
Consequently, the collection of real numbers $\{\mu((\det (W))^{\otimes j})\}_{j=1}^\infty$
is bounded above.

On the other hand, we have $\mu((\det (W))^{\otimes j})\,=\, js\cdot\mu(W)$. Since $\mu(W) \, 
>\, 0$, this contradicts the above observation that $\{\mu((\det (W))^{\otimes 
j})\}_{j=1}^\infty$ is bounded above. Therefore, $E$ is semistable.
\end{proof}

\begin{proposition}\label{prop3}
Let $E$ be a finite vector bundle on $M$, and let $Q$ be a torsionfree quotient of $E$
of degree zero. Then $Q$ is locally free.
\end{proposition}

\begin{proof}
The open subset of $M$ over which $Q$ is locally free will be denoted by $U$.
The complement $M\setminus U$ is a complex analytic subspace of $M$ of codimension at least two;
this is because $Q$ is torsionfree.

Let $r$ be the rank of $Q$. The quotient map $E\, \longrightarrow\, Q$ produces an
${\mathcal O}_M$--linear homomorphism
\begin{equation}\label{e2}
\varphi\, :\, \bigwedge\nolimits^r E
 \, \longrightarrow\, \det (Q)\, .
\end{equation}
Note that the restriction $\varphi\vert_U\, :\, 
(\bigwedge^r E)\vert_U \, \longrightarrow\, (\det (Q))\vert_U$ is surjective, because
we have $(\det (Q))\vert_U\,=\, \bigwedge^r (Q\vert_U)$.

We shall first show that the holomorphic line bundle $\det (Q)$ is of finite order.

The image $\varphi(\bigwedge\nolimits^r E)\, \subset\, \det (Q)$ will be denoted by $L$,
where $\varphi$ is the homomorphism in \eqref{e2}. 
Hence $L^j$ is a quotient of $(\bigwedge^r E)^{\otimes j}$. Since $(\bigwedge^r E)^{\otimes 
j}$ is a component of $E^{\otimes rj}$ (see Definition \ref{def0}), from Theorem \ref{thm1} 
and \eqref{e1} we conclude that $(\bigwedge^r E)^{\otimes j}$ is a direct sum of copies of 
$V_1,\, \cdots,\, V_p$ (recall that $V_1,\, \cdots,\, V_p$ are assumed
to be indecomposable). The vector 
bundles $V_1,\, \cdots,\, V_p$ are semistable of degree zero by Lemma \ref{lem1}, and hence 
using Proposition \ref{prop1} it follows that the stable quotients of degree zero of a direct 
sum of copies of $V_1,\, \cdots,\, V_p$ are just the stable quotients of degree zero of some 
$V_i$ occurring in the direct sum $\bigoplus_{i=1}^p V_i$.

Since $\text{degree}(L^j/\text{Torsion}) \,=\, j\cdot\text{degree}(\det (Q))\,=\, 0$, and
$L^j/\text{Torsion}$ is a quotient of a direct sum of copies of $V_1,\, 
\cdots,\, V_p$, we conclude that $L^j/\text{Torsion}$ is a quotient of some $V_i$, $1\,\leq\, 
i\,\leq\, p$. Each $V_i$ has only finitely many torsionfree quotients of rank one and degree 
zero (see Proposition \ref{prop1}); this can also be deduced from \cite[p.~1034, Proposition 
3.1]{BDL} (this proposition says that the semistable vector bundle $V^*_i$ admits only finitely many 
isomorphisms classes of reflexive stable subsheaves of degree zero). Consequently, we have 
$L^{\otimes a}\vert_U\,=\, L^{\otimes b}\vert_U$ for some $a\, \not=\, b$. This implies that 
$\det (Q)^{\otimes a}\,=\, \det(Q)^{\otimes b}$, because the complement $M\setminus U$ is a complex 
analytic subspace of $M$ of codimension at least two, and $L\vert_U\, =\, \det(Q)\vert_U$.
Therefore, the line bundle $\det (Q)$ is of finite order.

Since $\det (Q)$ is a finite bundle, from Proposition \ref{prop2} we conclude that the
homomorphism $\varphi$ in \eqref{e2} is surjective.

Consider the projective bundle ${\mathbb P}(\bigwedge\nolimits^r E)\, \longrightarrow\, M$
parametrizing the hyperplanes in the fibers of $\bigwedge\nolimits^r E$. Let $\text{Gr}(E)$
be the Grassmann bundle over $M$ parametrizing the $r$--dimensional quotients of the fibers
of $E$. Let
\begin{equation}\label{beta}
\beta\, :\, \text{Gr}(E)\, \longrightarrow\,{\mathbb P}(\bigwedge\nolimits^r E)
\end{equation}
be the Pl\"ucker embedding that sends a quotient $\widehat{q}_x \, :\, E_x\, 
\longrightarrow\, {\mathcal Q}_x$ of dimension $r$ to the kernel of the homomorphism 
$\bigwedge^r \widehat{q}_x \, :\, \bigwedge^r E_x\, \longrightarrow\, \bigwedge^r {\mathcal 
Q}_x$.

Let
$$
\sigma\, :\, M\, \longrightarrow\, {\mathbb P}(\bigwedge\nolimits^r E)
$$
be the holomorphic section
defined by the surjective homomorphism $\varphi$ in \eqref{e2}. The image of
the restriction $\sigma\vert_U$ lies in $\beta(\text{Gr}(E)\vert_U)\, \subset\,
{\mathbb P}(\bigwedge\nolimits^r E)$, where $\beta$ is the map in \eqref{beta}. Consequently,
$\sigma(M)$ lies in $\beta(\text{Gr}(E))$, because $\beta(\text{Gr}(E))$ is a closed submanifold
of ${\mathbb P}(\bigwedge\nolimits^r E)$.

Since $\sigma(M)$ lies in $\beta(\text{Gr}(E))$, it is evident that $Q$ is a quotient bundle of $E$.
\end{proof}

\section{Flat connection on finite bundles}

A holomorphic vector bundle $E$ over $M$ would be called \textit{\'etale trivial} if there is 
a finite connected \'etale covering
$$
\varpi\, :\, Y\, \longrightarrow\, M
$$
such that the vector bundle $\varpi^* E$ is holomorphically trivializable.

\begin{proposition}\label{prop4}
Let $E$ be a finite stable vector bundle on $M$. Then $E$ is
\'etale trivial.
\end{proposition}

\begin{proof}
Consider the vector bundles $V_1,\, \cdots,\, V_p$ in \eqref{e1}. As before, we assume that each
$V_i$ is indecomposable and it is a component of some $E^{\otimes j}$. From the given condition
that $E$ is stable
it follows that $E^{\otimes j}$ is polystable (see Corollary \ref{cor1}). Using Lemma \ref{lem1}
we conclude that the degree of $E^{\otimes j}$ is zero. Hence any component (see
Definition \ref{def0}) of $E^{\otimes j}$
is also polystable of degree zero. In particular, all $V_i$ are polystable of degree zero. Since each
$V_i$ is also indecomposable, it is stable of degree zero.

If $V_i$ is a component (see Definition \ref{def0}) of $E^{\otimes j}$, then $V^{\otimes n}_i$
is a component of $E^{\otimes nj}$, so $V^{\otimes n}_i$ is also a direct sum of copies of
$V_1,\, \cdots,\, V_p$.

Any nonzero homomorphism between two stable vector bundles of degree zero is an isomorphism. 
Take any reflexive subsheaf $S$ of degree zero of a direct sum $\mathcal U$ of copies of 
$V_1,\,\cdots,\, V_p$. Then $S$ is a subbundle of $\mathcal U$ by Proposition \ref{prop3}. 
Moreover, $S$ is a component of $\mathcal U$, because $\mathcal U$ is polystable of degree 
zero, and hence $S$ is a direct sum of copies of $V_1,\,\cdots,\, V_p$. Furthermore, any 
torsionfree quotient of degree zero of such a subsheaf $S$ is again a direct sum of copies of 
$V_1,\,\cdots,\, V_p$, because $S$ is polystable of degree zero. In particular, for any any homomorphism
$$
f\, :\, \bigoplus_{i=1}^p V^{\oplus c_i}_i\, \longrightarrow\, \bigoplus_{i=1}^p V^{\oplus d_i}_i\, ,
$$
both $\text{kernel}(f)$ and $\text{cokernel}(f)$ are direct sums of copies of $V_1,\,\cdots,\, V_p$.

Fix a point $x_0\, \in\, M$. Assign the vector space $F_{x_0}$ to any holomorphic vector
bundle $F$ on $M$. In view of the above observations, $V_1,\,\cdots,\, V_p$ produce a
Tannakian category (defined in \cite[p.~30]{No1}, \cite[p.~76]{No2}). Now using
\cite[p.~31, Theorem 1.1]{No1}, \cite[p.~77, Theorem 1.1]{No2}, this Tannakian category produces a
complex affine group-scheme. The isomorphism classes of indecomposable objects of this Tannakian category
are contained in the union of the following three:
\begin{enumerate}
\item $\{V_1,\, \cdots, \, V_p\}$,

\item $\{V^*_1,\, \cdots, \, V^*_p\}$, and

\item all indecomposable components (see Definition \ref{def0})
of all $V_i\otimes V^*_j$, $1\, \leq\, i,\, j\, \leq\, p$. Note that
$V_i\otimes V^*_j$ is polystable by Corollary \ref{cor1}; also, $\text{degree}(V_i\otimes V^*_j)
\,=\,0$, because $\text{degree}(V_i)\,=\,0$ for all $i$. So all indecomposable components of
$V_i\otimes V^*_j$ are stable of degree zero.
\end{enumerate}
Since the above union is a finite collection, from \cite[p.~31, Theorem 1.2]{No1},
\cite[p.~77, Theorem 1.2]{No2} it follows that the group-scheme defined by the above Tannakian category
is finite; this finite group will be denoted by $\Gamma$.

Fix a basis of $E_{x_0}$. Now the structure group $\text{GL}(r_E,{\mathbb C})$ of $E$,
where $r_E$ denotes the rank of $E$, has a reduction of structure group to this group $\Gamma$
\cite[p.~79, Theorem 2.9]{No2} (see also \cite[p.~34, Proposition 2.9]{No2}). Since our group
$\Gamma$ is finite, the proof of \cite[Theorem 2.9]{No2} remains valid without any
modification.

Therefore, there is a finite \'etale Galois covering $\varpi\, :\, \widehat{Y}\, 
\longrightarrow\, M$ with Galois group $\Gamma$ such that the vector bundle $\varpi^*E$ is 
holomorphically trivializable. Now taking $Y$ to be a connected component of $\widehat{Y}$ we 
conclude that $E$ is \'etale trivial.
\end{proof}

Let $E$ be a finite vector bundle over $M$. As in Proposition \ref{prop1}, let
\begin{equation}\label{e3}
0\,=\, F_0\, \subset\, F_1\, \subset\, \cdots \, \subset\, F_{m-1} \, \subset\, F_{m}\,=\, E
\end{equation}
be a filtration of reflexive subsheaves of $E$ such that $F_k/F_{k-1}$ is stable
of degree zero for all $1\, \leq\, k\, \leq\, m$. From Proposition \ref{prop3} we know
that each $F_k$ is a subbundle of $E$.

\begin{lemma}\label{lem2}
For every $1\, \leq\, k\, \leq\, m$, the holomorphic vector bundle $F_k/F_{k-1}$ in
\eqref{e3} is finite.
\end{lemma}

\begin{proof}
Consider the vector bundles $V_1,\, \cdots,\, V_p$ in \eqref{e1}. Since they are all finite, 
just as in \eqref{e3}, for every $1\, \leq\, i\, \leq\, p$ there is a filtration of 
holomorphic subbundles
\begin{equation}\label{e6}
0\,=\, U^i_0\, \subset\, U^i_1\, \subset\, \cdots \, \subset\, U^i_{n_i-1} \,
\subset\, U^i_{n_i}\,=\, V_i
\end{equation}
such that $U^i_\nu/U^i_{\nu-1}$ is a stable vector bundle of degree zero for all $1\,
\leq\, \nu\, \leq\, n_i$. Now consider the collection of stable vector bundles
\begin{equation}\label{u}
\{\{U^i_\nu/U^i_{\nu-1}\}^{n_i}_{\nu=1}\}^p_{i=1}\, .
\end{equation}
{}From Proposition \ref{prop1} it follows that this collection does not depend
on the choice of the filtrations in \eqref{e6}.

Consider the vector bundle $F_k/F_{k-1}$ in the statement of the proposition. Since 
$F_k/F_{k-1}$ is stable, by Corollary \ref{cor1}, the vector bundle $(F_k/F_{k-1})^{\otimes 
j}$ is polystable for every positive integer $j$.
Also, $$\mu((F_k/F_{k-1})^{\otimes j})\,=\, j\cdot \mu(F_k/F_{k-1})\, =\,0\, .$$

Since $F^{\otimes j}_k$ is a subbundle of $E^{\otimes j}$ of degree zero, and $E^{\otimes j}$ 
is semistable of degree zero (recall that $E^{\otimes j}$ is a finite bundle), it follows 
that $F^{\otimes j}_k$ is semistable of degree zero. Also, $(F_k/F_{k-1})^{\otimes j}$ is a 
polystable quotient of $F^{\otimes j}_k$ of degree zero. From these it follows that 
$(F_k/F_{k-1})^{\otimes j}$ is a direct sum of copies of stable graded factors of $E^{\otimes 
j}$ (see Definition \ref{def1}). Now using \eqref{e1} and Proposition \ref{prop1} we conclude that 
$(F_k/F_{k-1})^{\otimes j}$ is a direct sum of copies of the vector bundles in the collection 
in \eqref{u}. Therefore, $F_k/F_{k-1}$ is a finite bundle.
\end{proof}

\begin{theorem}\label{thm2}
A holomorphic vector bundle on $M$ is finite if and only if it is \'etale trivial.
\end{theorem}

\begin{proof}
Let $E$ be a finite vector bundle on $M$. Take a filtration of subbundles of it
$$
0\,=\, F_0\, \subset\, F_1\, \subset\, \cdots \, \subset\, F_{m-1} \, \subset\, F_{m}\,=\, E
$$
as in \eqref{e3}. From Lemma \ref{lem2} and Proposition \ref{prop4} we know that
$F_k/F_{k-1}$ is \'etale trivial for all $1\, \leq\, k\, \leq\, m$. Hence
there is a finite connected \'etale Galois covering
$$
\varpi\, :\, Y\, \longrightarrow\, M
$$
such that for the pulled back filtration
\begin{equation}\label{tf}
0\,=\, \varpi^*F_0\, \subset\, \varpi^*F_1\, \subset\, \cdots \, \subset\,
\varpi^*F_{m-1} \, \subset\, \varpi^*F_{m}\,=\, \varpi^*E\, ,
\end{equation}
the quotient bundle $(\varpi^*F_k)/(\varpi^* F_{k-1})\,=\, \varpi^* (F_k/F_{k-1})$ is 
trivializable for all $1\, \leq\, k\, \leq\, m$ (take a connected component of the fiber 
product of the \'etale Galois coverings for each $F_k/F_{k-1}$). Also, $\varpi^*E$ is finite 
because $E$ is so. 

Lemma 4.3 of \cite{BP} says that if a finite vector bundle $W$ admits a filtration of holomorphic
subbundles such that each successive quotient is holomorphically trivializable, then $W$ is
holomorphically trivializable. From this and \eqref{tf} it follows that $\varpi^*E$ is 
holomorphically trivializable. Hence $E$ is \'etale trivial.

To prove the converse, let $E$ be an \'etale trivial vector bundle on $M$. Let
$$
\varpi\, :\, Y\, \longrightarrow\, M
$$
be a connected \'etale covering such that $\varpi^*E$ is holomorphically trivializable. 
Consider the trivial connection $\nabla$ on $\varpi^*E$ corresponding to any holomorphic 
trivialization of $\varpi^*E$. This connection $\nabla$ does not depend on the choice of the 
holomorphic trivialization of $\varpi^*E$.

It can be shown that $\nabla$ descends to a connection on $E$. Indeed, by taking an \'etale
covering of $Y$ we may assume that $\varpi$ is Galois. The connection $\nabla$ is invariant
under the action of the Galois group $\text{Gal}(\varpi)$ on $\varpi^*E$, and hence
$\nabla$ descends to a connection on $E$. This 
descended connection on $E$ is holomorphic flat because $\nabla$ is holomorphic flat; also, 
its monodromy group is finite because $\varphi$ is a finite covering and the monodromy group 
of $\nabla$ is trivial. This immediately implies that $E$ is finite; see \cite[p.~7, Section 
4]{BP}.
\end{proof}

Theorem \ref{thm2} and its proof together have the following corollary.

\begin{corollary}\label{cor2}
A holomorphic vector bundle on $M$ is finite if and only if it admits a flat holomorphic
connection with finite monodromy.
\end{corollary}

\section*{Acknowledgements}

The author is grateful to Madhav Nori and Vamsi Pingali for helpful comments. He is
partially supported by a J. C. Bose Fellowship.


\end{document}